\documentclass[11pt]{amsart}
\usepackage[english]{babel}
\usepackage{amssymb}
\usepackage{amsmath}
\usepackage{booktabs}
\usepackage{url}
\usepackage{tikz}

\usepackage{xcolor}
\definecolor{darkred}{rgb}{0.45,0,0}
\definecolor{darkred}{rgb}{0.6,0,0}
\definecolor{darkblue}{rgb}{0,0,0.7}

\usepackage[
    colorlinks, 
    citecolor=blue, 
    urlcolor=darkred, 
    final, 
    hyperindex, 
    pagebackref,
    linkcolor = darkblue
]{hyperref}

\usepackage[capitalize]{cleveref}

\newcommand{\define}[1]{{\bf \boldmath{#1}}\index{#1}}

\newcommand{\ignore}[1]{}

\newtheorem{dummy}{Dummy}

\newtheorem{theorem}[dummy]{Theorem}
\newtheorem{proposition}[dummy]{Proposition}
\newtheorem{corollary}[dummy]{Corollary}

\theoremstyle{definition}

\newcommand{\maps}{\colon}    
\newcommand{\iso}{\xrightarrow{\raisebox{-0.2ex}{$\scriptstyle\sim$}}}
\newcommand{\tr}{\operatorname{tr}}
\newcommand{\Tr}{\operatorname{Tr}}

\newcommand{\R}{{\mathbb R}}  
\newcommand{\C}{{\mathbb C}}  
\newcommand{\Z}{{\mathbb Z}}  
\renewcommand{\H}{{\mathbb H}}  
\renewcommand{\O}{{\mathbb O}}  

\newcommand{\Aut}{{\rm Aut}}  
\newcommand{\SO}{{\rm SO}}    
\newcommand{\SU}{{\rm SU}}    
\newcommand{\SL}{{\rm{SL}}}  
\newcommand{\E}{{\rm E}}       
\newcommand{\G}{{\rm G}}       
\newcommand{\Hom}{\mathrm{Hom}}

\newcommand{\A}{\mathcal{A}}
\newcommand{\J}{\mathcal{J}}
\newcommand{\V}{\mathcal{V}}

\renewcommand{\sl}{\mathfrak{sl}} 
\newcommand{\e}{{\mathfrak{e}}}   
\newcommand{\g}{{\mathfrak{g}}}  

\newcommand{\M}{\mathrm{M}}   
\newcommand{\h}{\mathfrak{h}}  

\newcommand{\End}{{\rm End}} 


\author{John Baez}
\address{School of Mathematics\\
\\ University of Edinburgh, James Clerk Maxwell Building, Peter Guthrie Tait Road, Edinburgh, UK EH9 3FD.}
\email{baez@math.ucr.edu}

\author{Susanne Pumpl\"un}
\address{School of Mathematical Sciences\\
University of Nottingham\\
University Park\\
Nottingham NG7 2RD, United Kingdom}
\email{susanne.pumpluen@nottingham.ac.uk}

\begin{document}

\title{Three-dimensional Geometry \\ in Exceptional Mathematics}

\begin{abstract} 
We review some topics in ``exceptional mathematics'' from the perspective of 3-dimensional geometry: the octonions $\O$, the  split octonions $\O'$, the bioctonions $\O_\C \cong \C \otimes_\R \O$, the complex Albert algebra $\h_3(\O_\C)$, and the complex form of the exceptional Lie algebra $\e_6$.   We show how to functorially build an algebra isomorphic to $\O$ from any 3d complex vector space equipped with an inner product and complex volume form.  Similarly, we build one isomorphic to $\O'$ starting from a 3d real vector space equipped with a volume form, and one isomorphic to $\O_\C$ starting from a 3d complex vector space equipped with a complex volume form.  We give applications to 3-dimensional real and complex manifolds.   Finally, we describe how to build an Jordan algebra isomorphic to $\h_3(\O_\C)$ starting from three 3d complex vector spaces equipped with complex volume forms.  This last construction gives a nice explicit description of the complex Lie algebra $\e_6$ and its subalgebra $\sl(3,\C) \oplus \sl(3,\C) \oplus \sl(3,\C)$.
\end{abstract}

\maketitle

\section{Introduction}

The relation between the octonions $\O$ and 8-dimensional geometry are well known, and they have ripple effects into many other aspects of what one might call ``exceptional mathematics'' \cite{Ad,Ba,Y}.   But there is also a fascinating connection between the octonions equipped with a chosen square root of $-1$ and 3-dimensional complex geometry.  Just as we can describe the quaternions $\H$ as $\R \oplus \R^3$ equipped with a product built using operations familiar in 3-dimensional real geometry---the dot product and cross product---we can describe the octonions $\O$ as $\C \oplus \C^3$ equipped with a product built from analogous operations in 3-dimensional complex geometry.  

This construction privileges one embedding of algebras $\C \hookrightarrow \O$ among all the possible choices---or in other words, one square root of $-1$ among the whole 6-sphere of square roots of $-1$ in $\O$.  It thus breaks much of the symmetry inherent in the octonions.  However, this is just one of a number of constructions that relate 3-dimensional geometry to exceptional mathematics.   These constructions are not only systematic (i.e., functorial): they are reversible, so they provide equivalences between categories.  In summary, we have:

\begin{itemize}
\item Thm.\ \ref{thm:quaternion_functor}: The category of 3-dimensional oriented real inner product spaces is equivalent to the category of algebras isomorphic to the quaternions $\H$.  
\item Thm.\ \ref{thm:octonion_functor}: The category of 3-dimensional complex inner product spaces with normalized complex volume form is equivalent to the category of algebras isomorphic to the octonions $\O$ equipped with a chosen square root of $-1$.
\item Thm.\ \ref{thm:split_octonion_functor}: The category of 3-dimensional real vector spaces with volume form is equivalent to the category of algebras isomorphic to the split octonions $\O'$ equipped with a chosen nontrivial square root of $+1$.
\item Thm.\ \ref{thm:bioctonion_functor}: The category of 3-dimensional complex vector spaces with complex volume form is equivalent to the category of algebras isomorphic to the bioctonions $\O_\C$ equipped with a chosen nontrivial square root of $-1$.
\end{itemize}

We give some easy applications of these results to the geometry of real and complex 3-manifolds, but we expect that we are just skimming the surface.  As another sort of application---the one that originally motivated us to write this paper---we discuss the relation between the complex Albert algebra $\h_3(\O_\C)$ and 3-dimensional geometry.

Algebraists talk about Jordan algebras like loving teachers talk about their students: they are all either special or exceptional.  Any associative algebra $A$ becomes a Jordan algebra with product $a \circ b = \tfrac{1}{2} (a  b + b a)$.  Any Jordan subalgebra of one of this type is called ``special''.  All the rest are called ``exceptional''.  
Over $\C$, the only \emph{simple} exceptional Jordan algebra is the complex Albert algebra
\[ \h_3(\O_\C) \; = \;\; \left\{  \left( \begin{array}{ccc}  
                         \alpha  &  z  & y^*    \\  
                         z^*       & \beta & x      \\ 
                         y       & x^*   & \gamma   
\end{array} \right) : \; \alpha,\beta,\gamma \in \C, \; x,y,z \in \O_\C \right\} .\]
Here $\ast \maps \O_\C \to \O_\C$ is the complex-linear extension of $\ast \maps \O \to \O$ from $\O \subset \O_\C$ to $\O_\C$.  

However, there is another way to obtain an isomorphic Jordan algebra by putting a suitable product on
\[  \J = \M_3(\C) \oplus \M_3(\C) \oplus \M_3(\C).\]
This construction was developed by Tits \cite{Tits} in 1966, and later developed by McCrimmon \cite{McC69} and others.  It is a special case of what is now called the ``first Tits construction''.  The existence of an isomorphism between $\J$ and $\h_3(\O_\C)$ is well known \cite{PR86}.  An explicit formula for such an isomorphism is less well known, but can be found in Br\"uhne's thesis \cite{Br}.  In Theorem \ref{thm:iso}, we give a self-contained treatment of this isomorphism $\Phi \maps \J \to \h_3(\O_\C)$.  

In Section \ref{sec:e_6} we conclude with an application to the 78-dimensional exceptional Lie group $\E_6$.  The complex form of this group consists of all linear transformations of the complex Albert algebra preserving its cubic norm.   If we describe the complex Albert algebra as $\J$, we obtain a construction of the complex Lie algebra $\e_6$ that is a veritable  three-ring circus of 3-dimensional geometry.  This construction also gives access to an important inclusion of Lie algebras
\[   \sl(3,\C) \oplus \sl(3,\C) \oplus \sl(3,\C) \hookrightarrow \e_6 ,\]
which we describe in detail.

\subsection*{Terminology}

In mathematical physics $\h_3(\O)$ is often called ``the'' Albert algebra or ``the'' exceptional Jordan algebra, just as $\O$ is called ``the'' octonions.   Algebraists, on the other hand, have generalized these concepts dramatically: they define octonion algebras and Albert algebras axiomatically over an arbitrary field, commutative ring \cite{Al21,GPR,Pe19}, algebraic variety \cite{PST,Pu08}, or even a locally ringed space \cite{Ach1,Ach2}.  Even the least of these generalizations, working over a field, forces a change an attitude, since there can be several octonion algebras or Albert algebras over a given field. 

In this paper we work solely over $\R$ and $\C$.  All the above nuances evaporate over $\C$, because the complex numbers are algebraically closed: over $\C$ there is just one octonion algebra, the bioctonions $\O_\C$, and just one Albert algebra, the complex Albert algebra $\h_3(\O_\C)$.  But we see a tiny taste of these nuances over $\R$, because there we cannot trade a square root of $-1$ for a  square root of $+1$ by multiplying by $i$.  Over $\R$ there are thus two octonion algebras, the octonions $\O$ and split octonions $\O'$.  There are also three Albert algebras over $\R$, but we do not discuss these here.

In what follows, all vector spaces will be assumed finite-dimensional.  
By an \define{algebra} we mean a vector space $A$ with a bilinear map called the multiplication or product, often denoted simply by juxtaposition, and a \define{unit} $1 \in A$ such that $1x=x1=x$ for all $x\in A$. 

\section{Quaternions and octonions}
\label{sec:octonions}

By the octonions $\O$ we mean the unique 8-dimensional normed division algebra over $\R$.  There are many ways to construct this algebra \cite{Ba}, but here we describe one based on their relation to 3-dimensional complex geometry.   This approach copies ideas from the quaternions, $\H$: the unique 4-dimensional normed division algebra over $\R$.

A standard way to construct the quaternions takes $\H =  \R \oplus \R^3$
with the product given by 
\begin{equation}
\label{eq:quaternion_product}
(\alpha , \mathbf{a})(\beta , \mathbf{b}) = (\alpha \beta -  \mathbf{a}, \cdot \mathbf{b}, \; \alpha \mathbf{b} + \beta \mathbf{a} + \mathbf{a} \times \mathbf{b}) .
\end{equation}
This construction uses only $\SO(3)$-equivariant bilinear operations on scalars and 3-dimensional vectors, and it uses essentially all of them.  This is why quaternions were so useful before Gibbs and Heaviside isolated the dot product and cross product \cite{Crowe}.  This also instantly implies that $\SO(3)$ acts as automorphisms of the quaternions.  In fact, $\Aut(\H) = \SO(3)$.

Copying this story, we can build the octonions as $\O = \C \oplus \C^3$ equipped with the product
\[  (\alpha , \mathbf{a})(\beta , \mathbf{b}) = (\alpha \beta -  \langle \mathbf{a}, \mathbf{b}\rangle, \; \alpha \mathbf{b} + \beta \mathbf{a} + \mathbf{a} \; \overline{\!\times\!} \; \mathbf{b}) .\]
Here $\langle \mathbf{a}, \mathbf{b}\rangle$ is the usual inner product of vectors in $\C^3$, while $\mathbf{a} \; \overline{\!\times\!} \; \mathbf{b}$ is the componentwise complex conjugate of the cross product $\mathbf{a} \times \mathbf{b}$.  

\begin{proposition}
\label{prop:octonions_from_complex_vectors}
If we define multiplication on $\C \oplus \C^3$ by
\begin{equation}
\label{eq:octonion_product}
(\alpha , \mathbf{a})(\beta , \mathbf{b}) = (\alpha \beta -  \langle \mathbf{a}, \mathbf{b}\rangle, \; \alpha \mathbf{b} + \beta \mathbf{a} + \mathbf{a} \; \overline{\!\times\!} \; \mathbf{b}) 
\end{equation}
then this space becomes a 8-dimensional normed division algebra over the real numbers, which is therefore isomorphic to the octonions.
\end{proposition}

\begin{proof}
A version of this result over arbitrary commutative rings can be found in Section 2.2 of \cite{Pu24}.  The result at just the present level of generality is Lemma 3 of \cite{BS}.
\end{proof}

The inner product and conjugated cross product $\; \overline{\! \times \!} \; $ are $\SU(3)$-equivariant operations on $\C^3$, so $\SU(3)$ acts as algebra automorphisms of $\O = \C \oplus \C^3$.  Unlike the quaternion case, $\SU(3)$ is not the whole automorphism group of $\O$, which is the larger 14-dimensional group $\G_2$.  Instead, $\SU(3)$ gives only all the automorphisms of $\O$ that act trivially on the chosen subalgebra $\C \subset \O = \C \oplus \C^3$.  Indeed, the subgroup of $\Aut(\O)$ preserving any chosen square root of $-1$ is isomorphic to $\SU(3)$ \cite[Thm.\ 1.9.1]{Y}.

We can state these constructions of the quaternions and octonions in a more modern way as follows.  First consider the quaternionic case.  Suppose $V$ is a 3-dimensional real vector space with an inner product $\cdot \maps V \times V \to \R$ and orientation.  These structures determine a volume form $\omega \in \Lambda^3 V^\ast$, which turn lets us define an antisymmetric bilinear map 
\[   \times \maps V \times V \to V \]
by
\[   (\mathbf{a} \times \mathbf{b}, \mathbf{c}) = \omega(\mathbf{a},\mathbf{b},\mathbf{c})  \]
for all $\mathbf{a},\mathbf{b},\mathbf{c} \in V$. We can then make $\R \oplus V$ into an algebra using Equation \eqref{eq:quaternion_product}, and this algebra is isomorphic to $\H$.  Polishing this result still further, we obtain:

\begin{theorem}
\label{thm:quaternion_functor}
There is a functor $H$ from the category of 
\begin{itemize}
\item 3-dimensional oriented real inner product spaces, and orientation-preserving isometries
\end{itemize}
to the category of
\begin{itemize}
\item real algebras isomorphic to the quaternions, and algebra isomorphisms
\end{itemize}
sending $V$ to $H(V) = \R \oplus V$ equipped with the product in Equation \eqref{eq:quaternion_product}.  This functor $H$ is an equivalence of categories.
\end{theorem}


\begin{proof}
It is easy to check that an orientation-preserving isometry $f \maps V \to V'$ between 3-dimensional oriented real inner product spaces induces an algebra isomorphism $H(f) \maps \R \oplus V \to \R \oplus V'$, and that $H(fg) = H(f) H(g)$.   Thus, $H$ becomes a functor.  To show that $H$ is an equivalence, note that we can recover a 3-dimensional oriented real inner product space $V$ from any algebra $A \cong \H$  as follows.   First, there is a functorial splitting $A \cong \R \oplus V$ where $\R$ is spanned by $1 \in A$ and $V$ is spanned by the square roots of $-1$.  Second, the 3-dimensional space $V$ has a unique inner product such that the unit sphere consists of square roots of $-1$.  Third, for any orthornormal basis $i,j,k \in V$ we must have $i j = \pm k$, and if we say the basis is right-handed when $i j = k$, this determines an orientation on $V$.  This construction of an oriented 3-dimensional real inner product space from an algebra isomorphic to the quaternions is an inverse, up to natural isomorphism, to the functor $H$.
\end{proof}

The fact that $H$ is an equivalence is a way of making precise the idea that Hamilton captured essentially all the structure of a 3-dimensional oriented real inner product space with his invention of quaternions.  We can also apply the functor $H$ to a vector bundle whose fibers are 3-dimensional oriented real inner product spaces and obtain a bundle of algebras isomorphic to $\H$.  Thus, we obtain results like this:   

\begin{corollary}
For any oriented Riemannian 3-manifold $M$, there is a bundle of algebras $H(TM)$ whose fibers are algebras isomorphic to the split octonions, obtained by applying the functor $Z$ of Thm.\ \ref{thm:quaternion_functor} to the tangent bundle $T M$.
\end{corollary}

\noindent For a similar idea in a purely algebraic context see \cite{Pu98}.

Now let us carry out a similar program for the octonions, which are the example of main interest here.   Suppose $V$ is any 3-dimensional complex vector space with a complex inner product $\langle -, - \rangle \maps V \times V \to \C$.  This determines an isomorphism $V^\ast \cong V$ and an inner product on each grade $\Lambda^n V^\ast$ of the exterior algebra of $V^\ast$.   Suppose we choose an element $\omega \in \Lambda^3 V^\ast$ with $\langle \omega, \omega \rangle = 1$; we call $\omega$ a \define{normalized complex volume form} on $V$.  There is then an antisymmetric conjugate-bilinear map 
\[   \overline{\times} \maps V \times V \to V \]
determined by 
\[   \langle \mathbf{a} \, \overline{\times} \, \mathbf{b}, \mathbf{c} \rangle =  \omega(\mathbf{a}, \mathbf{b}, \mathbf{c}) .\]
We can make $\C \oplus V$ into an algebra using Equation \eqref{eq:octonion_product}, and this algebra is isomorphic to $\O$.
 Furthermore:


\begin{theorem}
\label{thm:octonion_functor}
There is a functor $O$ from the category of 
\begin{itemize}
\item 3-dimensional complex inner product spaces with normalized complex volume form, and unitary operators preserving the volume form
\end{itemize}
to the category of
\begin{itemize}
\item real algebras isomorphic to the octonions, equipped with chosen square root of $-1$, and algebra isomorphisms preserving this chosen square root of $-1$
\end{itemize}
sending $V$ to $\C \oplus V$ given the product in Equation \eqref{eq:octonion_product}.   This functor $O$ is an equivalence of categories.
\end{theorem}

\begin{proof}
The construction of the functor $O$ is similar to the proof of Thm.\ \ref{thm:quaternion_functor}, but we use some standard facts about the octonions \cite{Br}.  Note that for any algebra $A$, choosing a square root of $-1$ in $A$ is the same as choosing a monomorphism $\C \hookrightarrow A$.  To check that $O$ is an equivalence, we must recover a 3-dimensional complex inner product space $V$ with normalized volume form from any algebra $A \cong \O$ equipped with a monomorphism $\C \hookrightarrow \O$.   There is a unique real inner product $(-,-)$ on $A$ such that the unit sphere of $A$ contains $1$ and all square roots of $-1$ in $A$.  Using this we can split $A$ as $\C \oplus V$ where $\C$ is the chosen subalgebra and $V = \C^\perp$.  Left multiplication by any element of the subalgebra $\C$ preserves $V$, so $V$ becomes a complex vector space, and there is then a unique complex inner product $\langle - , - \rangle$ on $V$ such that $\langle v, w \rangle = (v,w) + i(iv,w)$ for all $v, w \in V$.   (Note we use the convention where $\langle - , - \rangle$ is complex-linear in the \emph{second} argument.)

The question then is how to recover the normalized volume form $\omega \in \Lambda^3 V^\ast$.   For this we set $\omega(u,v,w) = (uv, w)$ for all $u,v,w \in V$.   We can see that $\omega$ is antisymmetric in its first two arguments because $u,v,w \in V$ are ``pure imaginary'': that is, orthogonal to $1$ with respect to the real inner product.  This implies $u^2 = - (u, u)1$, and polarizing we obtain $u v + v u = -2 (u,v) 1$, and since also $w \in V$ we have the desired antisymmetry $(u v, w) + (v u, w) = 0$.  Moreover multiplication by pure imaginary elements is skew-adjoint, so $(u v, w) = -(u , w v)$.   Thus
\[   (u v , w) = -(v, u w) = -(u w, v) , \]
so $\omega$ is also antisymmetric in its last two arguments.  

Finally, we need to show $\omega$ is normalized. For any complex-orthonormal basis $j,k,\ell$ of $V$, the induced complex inner product on $\Lambda^3 V^\ast$ gives
\begin{equation}
\label{eq:omega_norm_squared}
  \langle \omega, \omega \rangle = |\omega(j,k,\ell)|^2,
\end{equation}
since $j,k,\ell$ orthonormal makes the alternating form $\langle j, - \rangle \wedge \langle k, - \rangle \wedge \langle \ell, - \rangle$ a unit vector spanning $\Lambda^3 V^\ast$.   Now $\omega(j,k,\ell) = (jk, \ell)$, and by Equation \eqref{eq:octonion_product}
the vector part of $j k$ is $j \mathbin{\overline{\times}} k$, so
\[
  \omega(j,k,\ell) = \langle j \, \overline{\times} \, k, \ell \rangle .
\]
Multiplicativity of the octonion norm applied to Equation \eqref{eq:octonion_product} gives
\[
  \langle j, k \rangle^2 + |j \mathbin{\overline{\times}} k|^2
   = |j|^2 |k|^2 .
\]
Since $j$ and $k$ are orthogonal unit vectors, this implies $|j \mathbin{\overline{\times}} k| = 1$.  Moreover $j \mathbin{\overline{\times}} k$ is orthogonal to $k$ with respect to the complex inner product, because
\[   \langle j \, \overline{\times} \, k, k \rangle = \omega(j,k,k) = 0. \]
Similarly $j \mathbin{\overline{\times}} k$ is orthogonal to $j$.
Since $j \mathbin{\overline{\times}} k$ is a unit vector orthogonal to both $j$ and $k$, it equals $e^{i\theta} \ell$ for some $\theta \in \R$, so 
\[  \omega(j,k,\ell) = \langle j \, \overline{\times} \, k, \ell \rangle =  e^{-i\theta}, \]
and $\langle \omega, \omega \rangle = 1$ by Equation \ref{eq:omega_norm_squared}.
\end{proof}

Again we can apply this construction fiberwise to any suitable vector bundle.  Suppose $M$ is a smooth 6-dimensional manifold. An \define{almost complex structure} on $M$ is a smoothly varying choice of complex structure $J$ on the tangent spaces of $M$.   An almost complex structure on $M$ allows us to define a complex line bundle $\Lambda^{3,0} T^\ast M$ whose sections are complex differential forms of type $(3,0)$.  A nowhere vanishing section of this bundle is called a \define{complex volume form} on $M$.   An almost complex structure on $M$ together with a Riemannian metric $g$ satisfying $g(J-,J-) = g(-,-)$ is called an \define{almost hermitian structure} on $M$.  An almost hermitian structure on $M$ together with a complex volume form normalized so that $\langle \omega_p, \omega_p \rangle = 1$ for each $p \in M$ is called a \define{$\SU(3)$ structure} on $M$, because it serves to reduce the structure group of the tangent bundle to $\SU(3)$ \cite{Joyce}.  An $\SU(3)$-structure is just the right structure on $M$ to apply Thm.\ \ref{thm:octonion_functor} to each tangent space.  

\begin{corollary}
\label{cor:SU(3)_structure}
For any 6-dimensional manifold with an $\SU(3)$ structure, there is a bundle of algebras $O(TM)$ whose fibers are algebras isomorphic to the octonions.   This is obtained by applying the functor $O$ of Thm.\ \ref{thm:octonion_functor} to each fiber of the tangent bundle $TM$.
\end{corollary}

\section{Split octonion algebras}
\label{sec:split_octonion_algebras}

Algebraists have a general concept of ``octonion algebra'' over a field: it is an 8-dimensional unital but not necessarily associative algebra $A$ equipped with a \define{quadratic norm}: a nondegenerate quadratic form $N \maps A \to k$ that obeys
\[         N(ab) = N(a) N(b)   .\]
Of course, this usage conflicts with the term ``norm'' in analysis, but we will have to live with that.  Over $\R$ there are just two nonisomorphic octonion algebras: the octonions $\O$ and the split octonions $\O'$.  The norm for the octonions is positive definite, but the split octonions have a norm of has signature $(4,4)$: evenly split between positive and negative signs.  

The split octonions were discovered by Max Zorn \cite{Z} in 1935. He expressed them as matrices where the diagonal elements are real numbers and the off-diagonal elements are vectors in $\R^3$, multiplied in the following peculiar way:
\begin{equation}
\label{eq:zorn_product}
\begin{pmatrix} \alpha & \mathbf{a}\\ \mathbf{a}' & \alpha' \end{pmatrix}
\begin{pmatrix} \beta & \mathbf{b}\\ \mathbf{b}' & \beta' \end{pmatrix}
=
\begin{pmatrix}
 \alpha\beta+ \mathbf{a}\cdot\mathbf{b}' & \alpha\mathbf{b}+\beta'\mathbf{a}-\mathbf{a}'\times\mathbf{b}'\\
 \beta\mathbf{a}' +\alpha'\mathbf{b}'+\mathbf{a}\times\mathbf{b} & \mathbf{a}'\cdot\mathbf{b}+ \alpha'\beta'
\end{pmatrix}.
\end{equation}
In these terms the quadratic norm $n \maps \O' \to \R$ resembles a determinant:
\[    n\begin{pmatrix} \alpha & \mathbf{a}\\ \mathbf{a}' & \alpha' \end{pmatrix} = \alpha \alpha' - \mathbf{a} \cdot \mathbf{a'}.\]
This norm is invariant under all automorphisms of $\O'$.

One can easily generalize Zorn's construction to other fields $k$, giving what is called a \define{split octonion algebra} over $k$.  We mention this merely because it lets us cover two examples at once: $k = \R$ and $k = \C$.  To bring Zorn's construction up to date, let us define it starting with any 3-dimensional vector space $V$ equipped with suitable structure.  Zorn's product uses the dot product and cross product, so from our experience with the quaternions and octonions in Section \ref{sec:octonions}, it is natural to guess that we need to equip $V$ with a nondegenerate bilinear form 
\[    (-,-) \maps V \times V \to k \]
and a normalized volume form: that is, an element $\omega \in \Lambda^3 V^\ast$ with $(\omega,\omega) = 1$.  However, in 1993 Petersson \cite[Sec.\ 3]{Pe93} carefully studied Zorn's product and noticed that the bilinear form $(-,-)$ is not required to define it!   All we need on $V$ is a volume form.

This is a good excuse to review some of the geometric algebra that can be done on a vector space equipped with only a volume form, which Barnabei, Brini and Rota called a \define{Peano space} \cite{BBR}.  This theory is not limited to three dimensions.  So, for a moment suppose $V$ is an $n$-dimensional vector space equipped with a volume form, meaning a nonzero element $\omega \in \Lambda^n V^\ast$.   The exterior algebra on $V$ is an associative graded algebra with the usual exterior product
\[   \wedge \maps \Lambda^i V \times \Lambda^j V \to \Lambda^{i+j} V .\]
Using the volume form $\omega$ we can make $\Lambda V$ into an associative algebra in another way, using what Grassmann called the \define{regressive product} 
\[   \vee \maps \Lambda^{n-i} V \times \Lambda^{n-j} V \to \Lambda^{n-(i+j)} V .\]
The point is that the volume form gives rise to a linear bijection $\star \maps \Lambda V \to \Lambda V^\ast$ that flips the grading upside down:
\[     \star \maps \Lambda^i V \iso \Lambda^{n-i} V^\ast \]
and lets us define the regressive product by 
\[     \mu \vee \nu =  \star^{-1}  (\star \mu \wedge \star \nu) \]
where the exterior product here is that on $\Lambda V^\ast$.

How is $\star$ defined?   First, there is a nondegenerate bilinear pairing
\[  \lbrace -,-\rbrace \maps \Lambda^i V^\ast \times \Lambda^i V \to k \]
defined by requiring that
\[   \lbrace f_1 \wedge \cdots \wedge f_i,\, \mathbf{a}_1 \wedge \cdots \wedge \mathbf{a}_i \rbrace = f_1(\mathbf{a}_1) \cdots f_i(\mathbf{a}_i) \]
for all $f_1, \dots, f_i \in V^\ast$ and $\mathbf{a}_1, \dots, \mathbf{a}_i \in V$.  Then $\star$ is determined by the equation
\[   \lbrace \star(\mathbf{a}_1 \wedge \cdots \wedge \mathbf{a}_i), \,\mathbf{a}_{i+1} \wedge \cdots \wedge \mathbf{a}_n \rbrace = 
\omega(\mathbf{a}_1 \wedge \cdots \wedge \mathbf{a}_n)\]
for all $\mathbf{a}_1, \dots, \mathbf{a}_n \in V$.

With this technology in hand, let us return to consider a 3-dimensional vector space $V$ with a volume form $\omega$.   We now regard a Zorn matrix
\[ \begin{pmatrix} \alpha & \mathbf{a}\\ \mathbf{a}' & \alpha' \end{pmatrix}
\]
as a notation for an arbitrary element 
\[   (\alpha, \mathbf{a}, \mathbf{a'} , \alpha') \in \Lambda^0 V \oplus \Lambda^1 V \oplus \Lambda^2 V \oplus \Lambda^3 V  = \Lambda V \]
of the exterior algebra of $V$.  Then we define the Zorn product on $\Lambda V$ by expressing Equation \eqref{eq:zorn_product} in terms of the exterior product and regressive product:
\[
\begin{pmatrix} \alpha & \mathbf{a} \\ \mathbf{a}' & \alpha' \end{pmatrix}
\begin{pmatrix} \beta & \mathbf{b} \\ \mathbf{b}' & \beta' \end{pmatrix}
=
\begin{pmatrix}
\alpha\wedge\beta + \mathbf{a}\vee\mathbf{b}'
&
\alpha\wedge\mathbf{b} + \beta'\vee\mathbf{a} - \mathbf{a}'\vee\mathbf{b}'
\\[6pt]
\alpha'\vee\mathbf{b}' + \beta\wedge\mathbf{a}' + \mathbf{a}\wedge\mathbf{b}
&
\alpha'\vee\beta' + \mathbf{a}'\wedge \mathbf{b}
\end{pmatrix}
\]
While this does little to clarify the precise formula for the Zorn product, it shows that this product relies solely on the 3-dimensional space $V$ and its volume form.  

Let $Z(V,\omega)$ denote $\Lambda V$ made into an octonion algebra with its Zorn product, which depends on $\omega \in \Lambda^3 V^\ast$.   Let $e \in Z(V,\omega)$ denote the identity for the Zorn product, to distinguish it from the identity $1$ for the exterior product.  Since $\Lambda^3 V$ is one-dimensional, there is a unique \define{covolume form} $\pi \in \Lambda^3 V$ with $\lbrace \omega, \pi \rbrace = 1$.  From looking at the Zorn matrices we see $e = 1 + \pi$.  It is easy to check that $1$ and $\pi$ are nontrivial idempotents in $Z(V,\omega)$, meaning idempotents for the Zorn product that equal neither $0$ nor $e$.

\begin{theorem}
\label{thm:split_octonion_functor}
For any field $k$, there is a functor $Z$ from the category of 
\begin{itemize}
\item 3-dimensional vector spaces with volume form, and linear operators preserving the volume form
\end{itemize}
to the category of
\begin{itemize}
\item algebras isomorphic to the split octonion algebra over $k$, equipped with chosen nontrivial idempotent, and algebra isomorphisms preserving this chosen idempotent
\end{itemize}
sending $(V,\omega)$ to the Zorn algebra $Z(V,\omega)$ equipped with the nontrivial idempotent $\pi$.   This functor $Z$ is an equivalence of categories.
\end{theorem}

\begin{proof}
If $(V,\omega)$ and $(V', \omega')$ are 3-dimensional vector spaces with volume form then any linear map $f \maps V \to V'$ induces linear maps $f_\ast \maps \Lambda V \to \Lambda V'$ and $f^\ast \maps \Lambda V'^\ast \to \Lambda V^\ast$.  We say $f$ preserves the volume form if $f^\ast \omega' = \omega$.  In this case $f$ must be a bijection, and one can check that $f_\ast \maps Z(V,\omega) \to Z(V',\omega')$ is an algebra isomorphism sending the covolume form $\pi$ of $Z(V,\omega)$ to the covolume form $\pi'$ of $Z(V,\omega')$.

To check that $Z$ is an equivalence, we must first recover a 3-dimensional vector space $V$ with volume form $\omega$ from an algebra $A$ isomorphic to the split octonions equipped with a nontrivial idempotent, and then find a natural isomorphism $A \cong Z(V,\omega)$. 

By the structure theory of split octonion algebras, all nontrivial idempotents of $A$ lie in the same orbit under $\Aut(A)$ \cite[Lemma 5.5]{KS}.  Thus if $w$ denotes the standard volume form on $k^3$, we can choose an isomorphism $\alpha \maps A \iso Z(k^3,w)$ such that $\alpha$ maps $\pi$ to the standard covolume form $p$ on $k^3$.   As vector spaces we have $ Z(k^3,w) = \Lambda k^3$, and a Zorn product calculation shows that the grading on the exterior algebra can be described as follows:
\[  \begin{array}{ccl}
\Lambda^0 k^3 &=& \{a \in Z(k^3,\omega) \, \vert \; p a = 0, \; a p = 0 \} \\ [3pt]
\Lambda^1 k^3 &=& \{a \in Z(k^3,\omega) \, \vert \; p a = 0, \; a p = a \} \\ [3pt]
\Lambda^2 k^3 &=& \{a \in Z(k^3,\omega) \, \vert \; p a = a, \; a p = 0 \} \\ [3pt]
\Lambda^3 k^3 &=& \{a \in Z(k^3,\omega) \, \vert \; p a = a, \; a p = a \}. 
\end{array}
\]
Pulling back this grading along $\alpha$ we get a grading of $A$
(as a vector space, not an algebra), and we take $V$ to be its degree-1 part:
\[   V = A_1 = \{a \in A \, \vert \; \pi a = 0, \; a \pi = a \} .\]
Another Zorn product calculation shows that the standard volume form $w$ on $k^3$ pulls back along $\alpha$ to the volume form 
\[   \omega \maps V \times V \times V \to k \]
such that $\omega(a,b,c)$ is the degree-3 part of $(ab)c$:
\[
\omega(a,b,c) = \big[(ab)c\big]_3, \qquad a,b,c \in V,
\]
where $[\,\cdot\,]_3$ denotes the degree-$3$ part. 

We now must produce a natural isomorphism $A \cong Z(V,\omega)$. The idea is that the grading and the products among its pieces reconstruct the entire algebra intrinsically. Because $\pi$ is an idempotent, the four subspaces $A_0, A_1, A_2, A_3 \subseteq A$ are the joint eigenspaces of left and right multiplication by $\pi$, exactly as in the displayed characterization.  A Zorn product calculation in $Z(k^3,w)$, transported to $A$ by $\alpha$, shows that multiplication respects this grading in the sense that
\[
A_i \cdot A_j \ \subseteq\ A_{i+j} \ \oplus\ A_{\,i+j-3}
\]
where $A_m := 0$ for $m < 0$ or $m > 3$. In particular:
\begin{itemize}
\item $A_0$ and $A_3$ are each one-dimensional, and multiplication by their generators recovers the two scalar actions; indeed $A_0 = k\hspace{0.5pt}1$ and $A_3 = k\hspace{0.5pt}\pi$, giving canonical isomorphisms $\Lambda^0 V \iso A_0$ and $\Lambda^3 V \iso A_3$, the latter sending the covolume form to $\pi$.
\item For $a, b \in A_1 = V$, the product $ab$ lies in $A_0 \oplus A_2$, and its degree-$2$ component defines a surjective alternating bilinear map $V \times V \to A_2$, $(a,b) \mapsto [ab]_2$, which is the wedge product; since $\dim V = 3$, this induces a canonical isomorphism $\Lambda^2 V \iso A_2$.
\item Likewise $[(ab)c]_3 = \omega(a,b,c)$ defines a surjective alternating trilinear map $V \times V \times V \to A_3$, compatibly identifying $\Lambda^3 V \iso A_3$ and matching $\omega$ with the covolume pairing.
\end{itemize}
Assembling these with the tautological identification $\Lambda^1 V = V = A_1$, we obtain a canonical linear isomorphism
\[
\Phi \maps \Lambda V = \Lambda^0 V \oplus \Lambda^1 V \oplus \Lambda^2 V \oplus \Lambda^3 V \ \iso\ A_0 \oplus A_1 \oplus A_2 \oplus A_3 = A,
\]
built entirely from the multiplication of $A$ and the idempotent $\pi$.

It is immediate from the construction that $\Phi \maps Z(V,\omega) \to A$ is an algebra homomorphism.  Finally, we need to check that $\Phi$ is natural. The subspace $V = A_1$, the volume form $\omega$, and the identifications of $A_0, A_2, A_3$ with $\Lambda^0 V, \Lambda^2 V, \Lambda^3 V$ were all defined solely in terms of the product on $A$ and the chosen idempotent $\pi$. Hence any morphism in the target category---an algebra isomorphism $g \maps A \to A'$ with $g(\pi) = \pi'$---carries $A_1$ to $A_1'$ by a volume-form--preserving linear map $g_1 \maps (V,\omega) \to (V',\omega')$, and satisfies $\Phi' \circ (g_1)_\ast = g \circ \Phi$. Thus $\Phi$ is a natural isomorphism.
\end{proof} 

Note that over a field of characteristic $\ne 2$, nontrivial idempotents correspond bijectively to nontrivial square roots of one, meaning elements $j \ne \pm e$ with $j^2 = e$: any such $j$ gives a nontrivial idempotent $\pi = (e + j)/2$, and conversely any nontrivial idempotent $\pi$ gives a nontrivial square root of one, namely $j = 2\pi - e$.  So, except in characteristic $2$ we could restate Theorem \ref{thm:split_octonion_functor} in terms of square roots of one.

This theorem has implications both for the split octonion algebra over $\R$, which we are calling ``the'' split octonions $\O'$, and also for the split octonion algebra over $\C$, which we call the bioctonions $\O_\C$.  We postpone the latter to the next section, but here is a consequence for $\O'$:

\begin{corollary}
For any 3-dimensional manifold $M$ equipped with a nowhere vanishing 3-form $\omega$, there is a bundle of algebras $Z(TM,\omega)$ whose fibers are algebras isomorphic to the split octonions, obtained by applying the functor $Z$ of Thm.\ \ref{thm:split_octonion_functor} to each fiber $T_pM$ of the tangent bundle $TM$, equipped with its volume form $\omega_p$.
\end{corollary}

\section{The bioctonions}

Because the complex numbers are algebraically closed, there is up to isomorphism one octonion algebra over $\C$, which we call the \define{bioctonions} and denote as $\O_\C$.  We can obtain this by complexifying the octonions, so
\[  \O_\C \cong \C \otimes_\mathbb{R} \O \]
We can also obtain the bioctonions by complexifying the split octonions:
\[  \O_\C \cong \C \otimes_\mathbb{R} \O' .\]
The Zorn matrix picture from the last section gives an isomorphism
\[  \O'_\C \cong Z(\C^3, \omega) \]
where $w$ is the standard complex volume form on $\C^3$.   Thus, for computational purposes we take the following as our definition of bioctonions:
\[
\O_\C=
\left\{  \begin{pmatrix} \alpha & \mathbf{a}\\ \mathbf{a}' & \alpha' \end{pmatrix} \; \Big{\vert} \; \alpha, \alpha' \in \C, \; \mathbf{a}, \mathbf{a}' \in \C^3 \right\}
\]
with the Zorn matrix product as its multiplication:
\[ \begin{pmatrix} \alpha & \mathbf{a}\\ \mathbf{a}' & \alpha' \end{pmatrix} 
\begin{pmatrix} \beta & \mathbf{b}\\ \mathbf{b}' & \beta' \end{pmatrix}
=
\begin{pmatrix}
 \alpha\beta+ \mathbf{a}\cdot\mathbf{b}' & \alpha\mathbf{b}+\beta'\mathbf{a}-\mathbf{a}'\times\mathbf{b}'\\
 \beta\mathbf{a}' +\alpha'\mathbf{b}'+\mathbf{a}\times\mathbf{b} & \mathbf{a}'\cdot\mathbf{b}+ \alpha'\beta'
\end{pmatrix},
\]
where $\times \maps \C^3\times \C^3\to \C^3$ is the standard vector product on $\C^3$ and $\cdot \maps \C^3 \times \C^3 \to \C^3$ is the standard complex-bilinear dot product.   
The bioctonions have quadratic norm $n$ and trace $\tr$ given by
\begin{equation}
\label{eq:bioctonion_norm_and_trace}
n\begin{pmatrix} \alpha & \mathbf{a}\\ \mathbf{a}' & \alpha' \end{pmatrix}=
\alpha\alpha'-\mathbf{a}\cdot\mathbf{a}', \quad
\tr\begin{pmatrix} \alpha & \mathbf{a}\\ \mathbf{a}' & \alpha' \end{pmatrix}=
\alpha+\alpha'.
\end{equation}
The \define{conjugate} of an element of $\O_\C$ is obtained by swapping the diagonal entries and negating the off-diagonal entries:
\begin{equation}
\label{eq:bioctonion_conjugate}
\begin{pmatrix} \alpha & \mathbf{a}\\ \mathbf{a}' & \alpha' \end{pmatrix}^{\!\ast}
=
\begin{pmatrix} \alpha' & -\mathbf{a}\\ -\mathbf{a}' & \alpha \end{pmatrix} .
\end{equation}
Thus $x + x^\ast = 2 \tr(x)\,1$ and $x\,x^\ast = x^\ast x = n(x)\,1$, and $\ast$ is a complex-linear involution with $(xy)^\ast = y^\ast x^\ast$.

Our main application of this picture of the bioctonions will be to the complex Albert algebra $\h_3(\O_\C)$ in Section \ref{sec:complex_Albert_algebra}.  For now we just state some easy spinoffs of Thm.\ \ref{thm:split_octonion_functor}.  For the first, let us define a \define{nontrivial} square root of $-1$ in $\O_\C$ to be one other than $\pm i$ in the central copy of $\C$ in $\C \otimes_\R \O = \O_\C$.

\begin{theorem}
\label{thm:bioctonion_functor}
There is a functor $Z$ from the category of 
\begin{itemize}
\item 3-dimensional complex vector spaces with complex volume form, and linear operators preserving the complex volume form
\end{itemize}
to the category of
\begin{itemize}
\item algebras isomorphic to the bioctonions equipped with chosen nontrivial square root of $-1$, and algebra isomorphisms preserving this chosen square root of $-1$.
\end{itemize}
This functor $Z$ is an equivalence of categories.
\end{theorem}

\begin{proof} 
Thm.\ \ref{thm:split_octonion_functor} yields this result for complex algebras isomorphic to the bioctonions equipped with a nontrivial idempotent, but from a nontrivial idempotent $p$ in such an algebra we can construct a nontrivial square root of $-1$, namely $i(1 - 2p)$, and vice versa.  
\end{proof}

Following the by now familiar pattern, this can be applied to construct bioctonion bundles.  Recall from the discussion before Corollary \ref{cor:SU(3)_structure} that an almost complex structure on a 6-dimensional manifold $M$ allows us to define a complex line bundle $\Lambda^{3,0} T^\ast M$ whose sections are complex differential forms of type $(3,0)$.   A nowhere vanishing section of this bundle is called a \define{complex volume form} on $M$.

\begin{corollary}
For any 6-dimensional manifold $M$ equipped with an almost complex structure and complex volume form $\omega$, there is a bundle of algebras $Z(TM,\omega)$ whose fibers are algebras isomorphic to the bioctonions, obtained by applying the functor $Z$ of Thm.\ \ref{thm:split_octonion_functor} to each complex vector space $T_pM$ equipped with its complex volume form $\omega_p$.
\end{corollary}

The most charismatic examples of the above corollary arise from \define{Calabi--Yau three-folds}: 3-dimensional K\"ahler manifolds $M$ for which the canonical bundle, i.e.\ the bundle $\Lambda^{3,0} T^\ast M$, admits a holomorphic trivialization.  When compact, such a manifold admits a Ricci flat K\"ahler structure and a nowhere vanishing holomorphic section $\Omega$ of the canonical bundle whose norm is constant with respect to that metric.  In this case the construction above can be done in the holomorphic category.

\begin{corollary}
For any compact Calabi--Yau threefold $M$ equipped with a nonzero holomorphic section $\Omega$ of the canonical bundle of constant norm, there is a holomorphic bundle of algebras $Z(TM,\Omega)$ whose fibers are algebras isomorphic to the bioctonions.
\end{corollary}

Because the canonical bundle is trivial and this provides a copy of the algebra $\C$ in each fiber of $Z(TM,\omega)$, this bioctonion bundle comes equipped with a trivial subbundle of subalgebras isomorphic to $\C$.  However, the bioctonion bundle itself is typically nontrivial, since as a vector bundle it is nothing but the bundle of holomorphic differential forms on $M$.

\section{The complex Albert algebra }
\label{sec:complex_Albert_algebra}

Here we recall two constructions of the complex Albert algebra, and then give an explicit isomorphism between them using the Zorn matrix description of the bioctonions. 

In one construction, the complex Albert algebra is the space $\h_3(\O_\C)$ of $3\times 3$ self-adjoint matrices of bioctonions: that is, matrices
\begin{equation}\label{Y form}
Y = \begin{pmatrix} \xi_1 & x_3 & {x_2}^\ast \\ {x_3}^\ast & \xi_2 & x_1 \\ x_2 & {x_1}^\ast & \xi_3 \end{pmatrix}
\end{equation}
with $\xi_1,\xi_2,\xi_3 \in \C$ and $x_1,x_2,x_3 \in \O_\C$.  This algebra is
$27$-dimensional over $\C$, since each of the three octonion entries contributes $8$ dimensions while the diagonal entries contribute $3$.  It is a Jordan algebra with Jordan product
\[   a \circ b = \tfrac{1}{2} (a b + b a) .\]
It has a cubic norm $\operatorname{Det} \maps \h_3(\O_\C) \to \C$ given by 
\begin{equation}
\label{eq:FJ_norm}
\operatorname{Det}(Y) = \xi_1 \xi_2 \xi_3 - \xi_1 n(x_1) - \xi_2 n(x_2) - \xi_3 n(x_3) + \tr((x_1 x_2) x_3),
\end{equation}
where $n$ and $\tr$ are the norm and trace on $\O_\C$ as defined in Equation \ref{eq:bioctonion_norm_and_trace}. Note that
$\operatorname{Det}$ is not the determinant of a matrix over a commutative ring: Equation \eqref{eq:FJ_norm} serves as the definition. The bracketing in $(x_1x_2)x_3$ does not actually matter in this equation, as $\tr((x_1x_2)x_3) = \tr(x_1(x_2x_3))$.

Another approach to the complex Albert algebra uses the ``first Tits construction'' \cite{Tits}, \cite[Sec.\ 5]{McC69}, \cite[Sec.\ 4.5]{McC04}. In its simplest form, this takes a central simple associative algebra $\A$ of degree $3$ and a parameter $\mu$ and produces an Albert algebra $\J$. The associative algebra we use is $\A = \M_3(\C)$, and we take $\mu = 1$.  

On $\A = \M_3(\C)$ we have the determinant $\det \maps \A \to \C$, the trace $\tr \maps \A \to \C$, and the \define{adjugate} map $\sharp \maps \A \to \A$, characterised by
\begin{equation}\label{sharp inverse}
x^{\sharp}x = x x^{\sharp} = \det(x)\,1 .
\end{equation}
The adjugate is widely used in the formula for the inverse of a matrix in terms of cofactors, which amounts to $x^{-1} = x^\sharp \, \det(x)^{-1}$.  For $3 \times 3$ matrices the adjugate $x^\sharp$ is quadratic in $x$.

The Albert algebra provided by the first Tits construction is
\[
\J = \A_0 \oplus \A_1 \oplus \A_2, \qquad \A_i = \A \ \ (i = 0,1,2),
\]
of complex dimension $27$, with basepoint $1_\J = (1, 0, 0)$.  Many of the structures on $\A$ reappear at a higher level on $\J$---ultimately because $\J$ is itself isomorphic to a space of $3 \times 3$ matrices, namely $\h_3(\O_\C)$.  We introduce these by hand: we define a \define{cubic norm} $N \maps \J \to \C$, a \define{trace} $\Tr \maps \J \to \C$, and an \define{adjugate} $x \mapsto x^{\sharp} \in \J$
by
\begin{align}
N(A_0, A_1, A_2) &= \det(A_0) + \det(A_1) + \det(A_2) - \tr(A_0 A_1 A_2),
  \label{eq:Tits_norm}\\
\Tr(A_0, A_1, A_2) &= \tr(A_0), \label{eq:Tits_trace}\\
(A_0, A_1, A_2)^{\sharp} &= \big(A_0^{\sharp} - A_1 A_2,\ \
  A_2^{\sharp} - A_0 A_1,\ \ A_1^{\sharp} - A_2 A_0\big).
  \label{eq:Tits_sharp}
\end{align}
The norm $N$ is a cubic form, the trace $\Tr$ is linear, and the adjugate is quadratic.  To define the Jordan product on $\J$ we also need some auxiliary structures: we define a \define{trace bilinear form} $\Tr(-,-) \maps \J \times \J \to \C$ by
\begin{equation}
\label{eq:trace_form}
\Tr(x,y) = \Tr(x)\,\Tr(y) - \Tr(x \sharp y).
\end{equation}
and polarize the adjugate on $\J$ to get a bilinear \define{sharp map} $\sharp \maps \J \times \J \to \J$, given by 
\begin{equation}
\label{eq:sharp_map}
x \sharp y = (x + y)^{\sharp} - x^{\sharp} - y^{\sharp} .
\end{equation}

The data $(N, {}^{\sharp}, 1_\J)$ is an example of a `cubic norm structure', and every cubic norm structure determines a Jordan algebra in a uniform way \cite[Sec.\ 4.5]{McC04}, in which the Jordan product is given by
\begin{equation}\label{eq:cubic_jordan_product}
x \circ y = \tfrac{1}{2}\Big( x \sharp y + \Tr(x)\, y + \Tr(y)\, x
  - \big(\Tr(x)\Tr(y) - \Tr(x, y)\big)\, 1_\J \Big).
\end{equation}
This formula is not special to the Tits construction; it is the general recipe expressing the product in a cubic Jordan algebra in terms of its cubic norm \cite[Sec.\ 4.2]{McC04}.

Next we describe an explicit Jordan algebra isomorphism from the Albert algebra $\J$ to the Albert algebra $\h_3(\O_\C)$, found by Br\"uhne \cite[Prop.\ 3.6.1]{Br}.   Let $\Phi \maps \J \to \h_3(\O_\C)$ be the linear map that sends $A = (A_0, A_1, A_2) \in \J = \M_3(\C)^3$ to
\[ \Phi(A) = \begin{pmatrix} \xi_1 & x_3 & {x_2}^\ast \\ {x_3}^\ast & \xi_2 & x_1 \\ x_2 & {x_1}^\ast & \xi_3 \end{pmatrix} \]
where $\mathbf{e}_i $, $i=1,2,3$, is the canonical basis for $\C^3$, and we define
\begin{align*}
\xi_i &= (A_0)_{ii} \quad \text{for } i=1,2,3 \\
x_1 &= \begin{pmatrix} (A_0)_{23} & \sum (A_2)_{k1}\mathbf{e}_k \\ -\sum (A_1)_{1k}\mathbf{e}_k & (A_0)_{32} \end{pmatrix} \\
x_2 &= \begin{pmatrix} (A_0)_{31} & \sum (A_2)_{k2}\mathbf{e}_k \\ -\sum (A_1)_{2k}\mathbf{e}_k & (A_0)_{13} \end{pmatrix} \\
x_3 &= \begin{pmatrix} (A_0)_{12} & \sum (A_2)_{k3}\mathbf{e}_k \\ -\sum (A_1)_{3k}\mathbf{e}_k & (A_0)_{21} \end{pmatrix}.
\end{align*}
Thus, the diagonal of the matrix $A_0$ supplies the three diagonal entries $\xi_i$, the off-diagonal entries of $A_0$ supply the two scalar corners of each off-diagonal entry $x_j$, the $j$th column of $A_2$ supplies the upper right ($\mathbf{u}$) vector slot of $x_j$, and minus the $j$th row of $A_1$ supplies the lower left ($\mathbf{v}$) vector slot of $x_j$.

\begin{theorem} \label{thm:iso}
The linear map $\Phi \maps \J \to \h_3(\O_\C)$ is an isomorphism of Jordan algebras.
\end{theorem}

\begin{proof}
This is an expansion of Br\"uhne's argument \cite[Prop.\ 3.6.1]{Br}.   Since the calculations are lengthy, we defer the proof to Appendix \ref{sec:iso}.
\end{proof}

\section{The complex Lie algebra \texorpdfstring{$\e_6$}{e6}}
\label{sec:e_6}

As is well known \cite[Sec.\ 3.1]{Y}, the group of linear transformations of complex Albert algebra that preserve its cubic norm is the complex form of the exceptional Lie group $\E_6$.   In what follows, we give a construction of the Lie algebra $\e_6$ starting from a trio of 3-dimensional vector spaces.  This allows us to see, very concretely, an important Lie subalgebra of $\e_6$ isomorphic to $\sl(3,\C) \oplus \sl(3,\C) \oplus \sl(3,\C)$.   

One can also see the existence of this Lie subalgebra in other ways.   Borel and de Siebenthal showed that the maximal-rank Lie subalgebras of a complex simple Lie algebra can be classified with the help of its affine Dynkin diagram.  For $\e_6$ this diagram has a beautiful symmetry, reflecting the role of triality:

\vskip 0.5em
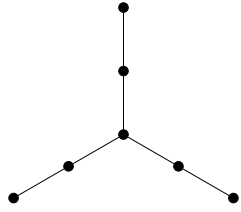
\begin{center}
\scalebox{0.6}{
\begin{tikzpicture}[
  every node/.style={draw,circle,fill=black,inner sep=0pt,minimum size=6pt}
]
  \node (c) at (0,0) {};
  \foreach \a in {90,210,330}{
    \node (n1-\a) at (\a:1.4) {};
    \node (n2-\a) at (\a:2.8) {};
    \draw (c) -- (n1-\a) -- (n2-\a);
  }
\end{tikzpicture}
}
\end{center}

\noindent Borel and de Siebenthal's theory says that if we remove any dot from this diagram, along with the edges incident to it, we get the Dynkin diagram of a maximal-rank Lie subalgebra of $\e_6$---i.e., a Lie subalgebra whose Cartan is 6-dimensional.  The most symmetrical option is to remove the central dot:

\vskip 0.5em
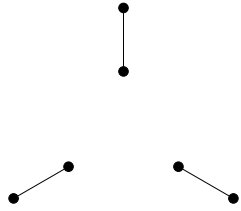
\begin{center}
\scalebox{0.6}{
\begin{tikzpicture}[
  every node/.style={draw,circle,fill=black,inner sep=0pt,minimum size=6pt}
]
  \foreach \a in {90,210,330}{
    \node (n1-\a) at (\a:1.4) {};
    \node (n2-\a) at (\a:2.8) {};
    \draw (n1-\a) -- (n2-\a);
  }
\end{tikzpicture}
}
\end{center}
This is the Dynkin diagram for the complex Lie algebra $\sl(3,\C) \oplus \sl(3,\C) \oplus \sl(3,\C)$.  

This Lie subalgebra is easy to find given the complex Albert algebra $\J$ and its cubic norm:

\begin{proposition}
\label{prop:SL(3)^3_action}
The group $\SL(3,\C)^3$ acts on $\J = M_3(\C)^3$, preserving the cubic norm $N$, as follows:
\[    R(g,h,k)(A_0,A_1,A_2) = (g A_0 h^{-1}, h A_1 k^{-1}, k A_2 g^{-1}). \]
This gives a homomorphism $R \maps \SL(3,\C)^3 \to \E_6$ whose kernel is 
\[  \big\{ (\omega I, \omega I, \omega I) \,\vert \; \omega^3 = 1 \big\} \cong \Z_3 .\]
\end{proposition}

\begin{proof}
The preservation of $N$ is evident from Equation \eqref{eq:Tits_norm}, which says $N(A_0, A_1, A_2) = \det(A_0) +  \det(A_1) +  \det(A_2) - \tr(A_0 A_1 A_2)$.  A transformation
\[     \begin{array}{ccc} \M_3(\C) &\to& \M_3(\C) \\
                              X   &\mapsto& aXb 
                              \end{array}
\]
for $a, b\in \SL(3,\C)$ only preserves the identity matrix if $a = b^{-1}$, and is then only the identity if $a$ is in the center of $\SL(3,\C)$, which implies that $a = \omega I$ for $\omega$ a cube root of unity.  Thus the kernel of $\rho$ is of the stated form.
\end{proof}

Since the kernel of $R$ is discrete, its differential $\rho = dR$ is an inclusion of Lie algebras
\[   \sl(3,\C)^3 \hookrightarrow \e_6 .\]

In fact one can describe the construction of the complex Albert algebra and the group of transformations preserving its cubic norm more functorially.  It is already clear from Section \ref{sec:complex_Albert_algebra} that for any 3-dimensional complex vector space $V$, the vector space $\End(V) \oplus \End(V) \oplus \End(V)$ with cubic norm
\[  N(A_0, A_1, A_2) = \det(A_0) + \det(A_1) + \det(A_2) - \tr(A_0 A_1 A_2) \]
is isomorphic to $\J$ with its cubic norm as in Equation \eqref{eq:Tits_norm}.  However, we can proceed more generally: we can start with \emph{three} 3-dimensional vector spaces.  This clarifies the appearance of $\SL(3,\C)^3$.

Suppose $\V = (V_0, V_1, V_2)$ is a triple of 3-dimensional complex vector spaces equipped with complex volume forms $\omega_i$.  We can build a version of the complex Albert algebra $\J$ starting from these vector spaces:
\[   \J_\V = \Hom(V_0, V_1) \oplus \Hom(V_1, V_2) \oplus \Hom(V_2, V_0) \]
and we can give it the cubic norm
\[   N_\V(A_0, A_1, A_2) = \det(A_0) + \det(A_1) + \det(A_2) - \tr(A_0 A_1 A_2) .\]
Here we use the fact that the determinant of a linear map between different vector spaces of the same dimension is well-defined when each is equipped with a volume form.   

Let $G(\J_\V)$ be the Lie group of invertible linear transformations of $\J_\V$ that preserve the cubic norm $N_\V$.    Clearly $\J_\V$ with its cubic norm $N_\V$ is isomorphic to $\J$ with the cubic form $N$ in Equation \eqref{eq:Tits_norm}.  Thus, $G(\J_\V)$ is isomorphic to the complex form of $\E_6$. 

It follows that the Lie algebra of $G(\J_\V)$ is isomorphic to the complex Lie algebra $\e_6$.   We call it $\e_6(\V)$ to indicate its dependence on the triple $\V$.  We now describe this Lie algebra explicitly as a $\Z_3$-graded Lie algebra with
\[
 \e_6(\V) \;=\; \g_0 \oplus \g_1 \oplus \g_2
\]
where
\[
\begin{array}{lcl}
  \g_0 &=& \sl(V_0)\oplus\sl(V_1)\oplus\sl(V_2), \\ [3pt]
  \g_1 &=& V_0\otimes V_1\otimes V_2 , \\ [3pt]
    \g_{-1} &=& V_0^{*}\otimes V_1^{*}\otimes V_2^{*}.
\end{array}
\]

To do this, we need two cross products.  These are special cases of the exterior and regressive products discussed in Section \ref{sec:split_octonion_algebras}, but expressed in terms of $V_i^\ast$ rather than the naturally isomorphic space $\Lambda^2 V_i$.  The first is a version of the exterior product $\wedge \maps V_i \times V_i \to \Lambda^2 V_i$.  For all $\mathbf{a},\mathbf{b} \in V_i$ we define $\mathbf{a}\times\mathbf{b}\in V_i^{*}$ by
\begin{equation}
\label{eq:ss}
  (\mathbf{a}\times\mathbf{b})(\mathbf{c}) \;=\; \omega_i(\mathbf{a},\mathbf{b},\mathbf{c}).
\end{equation}
for all $\mathbf{c} \in V_i$. The second is a version of the regressive product $\vee \maps \Lambda^2 V_i \times \Lambda^2 V_i \to V_i$.  For $g,h\in V_i^{*}$ we define $g\times h\in V_i$ by
\begin{equation}
\label{eq:s}
  \omega_i\bigl(\mathbf{a},\,\mathbf{b},\,g\times h\bigr) \;=\; g(\mathbf{a})\,h(\mathbf{b})-g(\mathbf{b})\,h(\mathbf{a})
\end{equation}
for all $\mathbf{a},\mathbf{b}\in V_i$.   Both cross products are bilinear and antisymmetric, and they obey the identities
\[
\begin{array}{ccl}
  (\mathbf{a}\times\mathbf{b})\times f &=& f(\mathbf{a})\,\mathbf{b}-f(\mathbf{b}),\\ [3pt]
  h\times (\mathbf{a}\times\mathbf{b}) &=& h(\mathbf{b})\,\mathbf{a}-h(\mathbf{a}).
\end{array}
\]
We define the bracket on $\e_6(\V)$ grade by grade, following Adams \cite[Chap.\ 13]{Ad} and Draper--Mart\'in \cite[Eq.\ 8]{DM}.

\medskip
\noindent \define{$[\g_0,\g_0]$.}
The three subspaces $\sl(V_i) \subset \g_0$ commute, and inside each one the bracket is the usual commutator:
\[
  [\phi,\psi]=\phi\psi-\psi\phi
\]
for $\phi,\psi\in\sl(V_i)$.

\medskip
\noindent \define{$[\g_0,\g_1]$ and $[\g_0,\g_2]$.}
The subspace $\sl(V_i)$ acts on the $i$th tensor slot by the natural (resp.\ dual) representation, trivially on the others.  For example, for $\phi_0 \in \sl(V_0)$ we have
\[
\begin{array}{ccr}
  [\phi,\;\mathbf{a}_0\otimes\mathbf{a}_1\otimes\mathbf{a}_2]
    &=&\phi_0(\mathbf{a}_0)\otimes \mathbf{a_1} \otimes\mathbf{a}_2, 
    \\ [3pt]
  [\phi_i,\;f_0\otimes f_1\otimes f_2]
    &=& -\phi_0^{\! *} (f_0) \otimes f_1 \otimes f_2,
\end{array}
\]
with $\phi_i^{\!*}$ the dual action on $V_i^{*}$.

\medskip
\noindent \define{$[\g_1,\g_1]$.}  We use the cross product in Equation \eqref{eq:ss}:
\[
  [\,\mathbf{a}_0\otimes\mathbf{a}_1\otimes\mathbf{a}_2,\ \mathbf{b}_0\otimes\mathbf{b}_1\otimes\mathbf{b}_2\,]
  = (\mathbf{a}_0\times\mathbf{b}_0)\otimes(\mathbf{a}_1\times\mathbf{b}_1)\otimes(\mathbf{a}_2\times\mathbf{b}_2) .\]

\medskip
\noindent \define{$[\g_2,\g_2]$.}
We use the cross product in Equation \eqref{eq:s}:
\[
  [\,f_0\otimes f_1\otimes f_2,\ g_0\otimes g_1\otimes g_2\,]
  = (f_0\times g_0)\otimes(f_1\times g_1)\otimes(f_2\times g_2).
\]

\medskip
\noindent \define{$[\g_1,\g_2]$.} For $u=\mathbf{a}_0\otimes\mathbf{a}_1\otimes\mathbf{a}_2$ and $\varphi=f_0\otimes f_1\otimes f_2$,
\[
  [\,u,\varphi\,]
  = \sum_{i\in \Z_3}\Bigl(\textstyle\prod_{j\neq i}f_j(\mathbf{a}_j)\Bigr)
        \Bigl(\mathbf{a}_i\otimes f_i-\tfrac13\,f_i(\mathbf{a}_i)\,1_{V_i}\Bigr),
\]
where $\mathbf{a}_i\otimes f_i\in V_i \otimes V_i^\ast \cong \End(V_i)$ denoteds the map
$\mathbf{y} \mapsto f_i(\mathbf{y})\,\mathbf{a}_i$. Since $\tr(\mathbf{a}_i\otimes f_i)=f_i(\mathbf{a}_i)$ and
$\dim V_i=3$, the subtracted term projects the $i$th summand into $\sl(V_i)$; the other slots $j\neq i$ multiply by the scalars $f_j(\mathbf{a}_j)$.

\medskip

Next we describe the representation of $\e_6(\V)$ on $\J_\V$:

\medskip

\define{Action of $\g_0$.}  The subalgebra $\g_0 = \sl(V_0) \oplus \sl(V_1) \oplus \sl(V_2)$ acts by pre-- and post--composition.  For $\phi_i\in\sl(V_i)$ and $T\in\Hom(V_a,V_b)$,
\begin{equation}
  \phi_i\cdot T \;=\;
  \begin{cases}
    \phi_b\,T, & i=b,\\[2pt]
    -\,T\,\phi_a, & i=a,\\[2pt]
    0, & i\notin\{a,b\}.
  \end{cases}
\end{equation}
(The index $i$ carried by $T$ is acted on in the two slots it occupies.)

\medskip

\define{Action of $\g_1$.}  For $u=\mathbf{a}_0\otimes\mathbf{a}_1\otimes\mathbf{a}_2\in\g_1$ and
$T\in \Hom(V_{i},V_{i+1})$, the image lies in
$\Hom(V_{i+1},V_{i+2})$ and is given by
\begin{equation}\label{eq:raise}
  (u\cdot T)(\mathbf{y})
  \;=\; \omega_{i+1}\!\bigl(\mathbf{a}_{i+1},\ T\mathbf{a}_{i},\ \mathbf{y}\bigr)\, \mathbf{a}_{i+2}
\end{equation}
for all $\mathbf{y}\in V_{i+1}$.

\medskip 

\define{Action of $\g_2$.}
For $\varphi=f_0\otimes f_1\otimes f_2\in\g_2$ and
$T\in \Hom(V_{i},V_{i+1})$, the image lies in
$\Hom(V_{i-1},V_{i})$ and is given by
\begin{equation}\label{eq:lower}
  (\varphi\cdot T)(\mathbf{a})
  \;=\; f_{i-1}(\mathbf{a})\,\bigl[(f_{i+1}\circ T)\times f_{i}\bigr],
  \qquad \mathbf{a}\in V_{i-1},
\end{equation}
where $f_{i+1}\circ T\in V_{i}^{*}$ and the cross product
$(f_{i+1}\circ T)\times f_{i}\in V_{i}$ is defined via Equation \eqref{eq:s}.
\medskip 

Just as $\e_6(\V)$ is $\Z_3$-graded, so is $\J_\V$, with
\[ 
  \J_\V^{(i)} = \Hom\bigl(V_{i},V_{i+1}\bigr).
\]
The representation of $\e_g(\V)$ on $\J_\V$ respects these $\Z_3$-gradings, with $\g_a$ mapping $\J^{(i)}$ to $\J^{(i+a)}$.

\section{Conclusion}

We have seen that 3-dimensional geometry pervades the mathematics of the octonions, split octonions, bioctonions, complex Albert algebra and $\E_6$.  However, the geometrical meaning of this fact remains somewhat obscure.  Is there a viewpoint from which it seems completely natural that these exceptional algebraic structures emerge from 3-dimensional geometry?  

\appendix
\section{\texorpdfstring{The isomorphism $\Phi \colon \J \to \h_3(\O_\C)$}{The isomorphism Phi: J -> h3(OC)}}
\label{sec:iso}

\addtocounter{dummy}{-2}
\begin{theorem} 
The linear map $\Phi \maps \J \to \h_3(\O_\C)$ is an isomorphism of Jordan algebras.
\end{theorem}

\begin{proof}
This is an expansion of Br\"uhne's argument \cite[Prop.\ 3.6.1]{Br}.    Thanks to the Springer construction \cite[Sec.\ 3.8]{McC04}, a linear bijection $\Phi \maps \J \to \h_3(\O_\C)$ satisfying $\Phi(1)=1$ is a Jordan isomorphism if $\operatorname{Det}(\Phi(X)) = N(X)$ for all $X$, where $N$ is the cubic norm of the first Tits construction and $\operatorname{Det}$ is the cubic norm \eqref{eq:FJ_norm} of $\h_3(\O_\C)$. Since $\Phi(1,0,0)=1$, it suffices to check that $\Phi$ preserves the cubic norm.

As described, for any triple $A = (A_0, A_1, A_2) \in \J$ we define 
\begin{equation}\label{Phi(X) form}
\Phi(A) = \begin{pmatrix} \xi_1 & x_3 & {x_2}^\ast \\ {x_3}^\ast & \xi_2 & x_1 \\ x_2 & {x_1}^\ast & \xi_3 \end{pmatrix} \in \h_3(\O_\C)
\end{equation}
where
\begin{equation}
\xi_i = (A_0)_{ii}
\end{equation}
and
\begin{equation}
x_1 = \begin{pmatrix} (A_0)_{23} & \mathbf{u}_1 \\ \mathbf{v}_1 & (A_0)_{32} \end{pmatrix}, \;
x_2 = \begin{pmatrix} (A_0)_{31} & \mathbf{u}_2 \\ \mathbf{v}_2 & (A_0)_{13} \end{pmatrix}, \;
x_3 = \begin{pmatrix} (A_0)_{12} & \mathbf{u}_3 \\ \mathbf{v}_3 & (A_0)_{21} \end{pmatrix}
\end{equation}
where the components of $\mathbf{u}_j$ and $\mathbf{v}_j$ are the columns of $A_2$ and minus the rows of $A_1$, respectively:
\begin{equation}
\mathbf{u}_j = \sum_{k=1}^3 (A_2)_{kj}\mathbf{e}_k, \quad \mathbf{v}_j = -\sum_{k=1}^3 (A_1)_{jk}\mathbf{e}_k.
\end{equation}

In what follows, we compute
\begin{equation}\label{Det to compute}
\operatorname{Det}(\Phi(A)) = \xi_1 \xi_2 \xi_3 - \xi_1 n(x_1) - \xi_2 n(x_2) - \xi_3 n(x_3) + \tr((x_1 x_2) x_3)
\end{equation}
term by term, and ultimately show that it equals the cubic norm $N(A)$ of Equation \eqref{eq:Tits_norm}.

The pure scalar diagonal product is $\xi_1 \xi_2 \xi_3 = (A_0)_{11} (A_0)_{22} (A_0)_{33}$.
Evaluating the component composition norms via the norm on $\O_\C$ yields
\begin{equation}
n(x_1) = (A_0)_{23}(A_0)_{32} - \mathbf{u}_1 \cdot \mathbf{v}_1 = (A_0)_{23}(A_0)_{32} + (A_1 A_2)_{11}.
\end{equation}
Multiplying by $\xi_1 = (A_0)_{11}$ and cyclically permuting for all three indices, the sum of these terms is
\begin{equation}\label{xi n sum}
\sum_{i=1}^3 \xi_i n(x_i) = \sum_{\text{cyclic}} (A_0)_{11}(A_0)_{23}(A_0)_{32} + \sum_{i=1}^3 (A_0)_{ii}(A_1 A_2)_{ii}.
\end{equation}

To compute the Zorn product
$x_1 x_2 = \begin{pmatrix} \alpha_{12} & \mathbf{a}_{12} \\ \mathbf{b}_{12} & \beta_{12} \end{pmatrix}$,
recall that
\[
x_1 = \begin{pmatrix} (A_0)_{23} & \mathbf{u}_1 \\ \mathbf{v}_1 & (A_0)_{32} \end{pmatrix},
\qquad
x_2 = \begin{pmatrix} (A_0)_{31} & \mathbf{u}_2 \\ \mathbf{v}_2 & (A_0)_{13} \end{pmatrix},
\]
so that in the notation of the Zorn matrix product rule  
\[
\begin{pmatrix} \alpha & \mathbf{a}\\ \mathbf{a}' & \alpha' \end{pmatrix}
\begin{pmatrix} \beta & \mathbf{b}\\ \mathbf{b}' & \beta' \end{pmatrix}
=
\begin{pmatrix}
 \alpha\beta+ \mathbf{a}\cdot\mathbf{b}' & \alpha\mathbf{b}+\beta'\mathbf{a}-\mathbf{a}'\times\mathbf{b}'\\
 \beta\mathbf{a}' +\alpha'\mathbf{b}'+\mathbf{a}\times\mathbf{b} & \mathbf{a}'\cdot\mathbf{b}+ \alpha'\beta'
\end{pmatrix}
\]
we have
\[
\alpha = (A_0)_{23},\quad \mathbf{a}=\mathbf{u}_1,\quad \mathbf{a}'=\mathbf{v}_1,\quad \alpha'=(A_0)_{32},
\]
\[
\beta = (A_0)_{31},\quad \mathbf{b}=\mathbf{u}_2,\quad \mathbf{b}'=\mathbf{v}_2,\quad \beta'=(A_0)_{13}.
\]
and thus
\begin{align}
\alpha_{12} &= (A_0)_{23}(A_0)_{31} + \mathbf{u}_1 \cdot \mathbf{v}_2 \label{alpha raw}\\
\mathbf{a}_{12} &= (A_0)_{23}\mathbf{u}_2 + (A_0)_{13}\mathbf{u}_1 - \mathbf{v}_1 \times \mathbf{v}_2 \\
\mathbf{b}_{12} &= (A_0)_{31}\mathbf{v}_1 + (A_0)_{32}\mathbf{v}_2 + \mathbf{u}_1 \times \mathbf{u}_2 \\
\beta_{12} &= \mathbf{v}_1 \cdot \mathbf{u}_2 + (A_0)_{32}(A_0)_{13} . \label{beta raw}
\end{align}
Next we express the two dot products above in terms of $A_1$ and $A_2$. Since
$(\mathbf{u}_j)_k = (A_2)_{kj}$ and $(\mathbf{v}_j)_k = -(A_1)_{jk}$, every such dot
product contracts a row index of $A_1$ with a column index of $A_2$:
\begin{equation}\label{dot products}
\begin{array}{c}
\mathbf{u}_i \cdot \mathbf{v}_j = -\displaystyle{\sum_{k=1}^3} (A_1)_{jk}(A_2)_{ki} = -(A_1A_2)_{ji}, \\ [5pt]
\mathbf{v}_i \cdot \mathbf{u}_j = -\displaystyle{\sum_{k=1}^3} (A_1)_{ik}(A_2)_{kj} = -(A_1A_2)_{ij} .
\end{array}
\end{equation}
Only entries of $A_1A_2$ can occur here, never entries of $A_2A_1$.
Applying \eqref{dot products} to \eqref{alpha raw} and \eqref{beta raw} yields
\begin{align}
\alpha_{12} &= (A_0)_{23}(A_0)_{31} - (A_1 A_2)_{21} \\
\beta_{12} &= (A_0)_{32}(A_0)_{13} - (A_1 A_2)_{12} .
\end{align}
Next we multiply $x_1x_2$ by $x_3$. In the notation of the product rule the two factors are
\[
x_1x_2 = \begin{pmatrix} \alpha_{12} & \mathbf{a}_{12}\\ \mathbf{b}_{12} & \beta_{12}\end{pmatrix},
\qquad
x_3 = \begin{pmatrix} (A_0)_{12} & \mathbf{u}_3\\ \mathbf{v}_3 & (A_0)_{21}\end{pmatrix},
\]
so $\alpha = \alpha_{12}$, $\mathbf{a}=\mathbf{a}_{12}$, $\mathbf{a}'=\mathbf{b}_{12}$,
$\alpha'=\beta_{12}$ and $\beta=(A_0)_{12}$, $\mathbf{b}=\mathbf{u}_3$,
$\mathbf{b}'=\mathbf{v}_3$, $\beta'=(A_0)_{21}$. Only the diagonal entries of the
product are needed. They are
\begin{align}
\big((x_1x_2)x_3\big)_{11} &= \alpha_{12}(A_0)_{12} + \mathbf{a}_{12}\cdot\mathbf{v}_3, \\
\big((x_1x_2)x_3\big)_{22} &= \mathbf{b}_{12}\cdot\mathbf{u}_3 + \beta_{12}(A_0)_{21},
\end{align}
and adding them gives
\begin{equation}\label{trace split}
\tr((x_1 x_2) x_3) = \alpha_{12}(A_0)_{12} + \mathbf{a}_{12}\cdot\mathbf{v}_3 + \beta_{12}(A_0)_{21} + \mathbf{b}_{12}\cdot\mathbf{u}_3.
\end{equation}
It remains to expand the two dot products above. Substituting the expressions for
$\mathbf{a}_{12}$ and $\mathbf{b}_{12}$ found above, we obtain
\begin{align}
\mathbf{a}_{12}\cdot\mathbf{v}_3 &= (A_0)_{23}(\mathbf{u}_2\cdot\mathbf{v}_3) + (A_0)_{13}(\mathbf{u}_1\cdot\mathbf{v}_3) - (\mathbf{v}_1\times\mathbf{v}_2)\cdot\mathbf{v}_3, \label{a dot v}\\
\mathbf{b}_{12}\cdot\mathbf{u}_3 &= (A_0)_{31}(\mathbf{v}_1\cdot\mathbf{u}_3) + (A_0)_{32}(\mathbf{v}_2\cdot\mathbf{u}_3) + (\mathbf{u}_1\times\mathbf{u}_2)\cdot\mathbf{u}_3. \label{b dot u}
\end{align}
Four of the dot products here can be computed using \eqref{dot products}:
\begin{equation}\label{four dots}
\begin{array}{l}
\mathbf{u}_2\cdot\mathbf{v}_3 = -(A_1A_2)_{32}, \\
\mathbf{u}_1\cdot\mathbf{v}_3 = -(A_1A_2)_{31}, \\
\mathbf{v}_1\cdot\mathbf{u}_3 = -(A_1A_2)_{13}, \\
\mathbf{v}_2\cdot\mathbf{u}_3 = -(A_1A_2)_{23}.
\end{array}
\end{equation}
For the two scalar triple products, recall that the determinant of a $3\times3$
matrix is the triple product of its columns, and equally of its rows. Since
$\mathbf{u}_1,\mathbf{u}_2,\mathbf{u}_3$ are the columns of $A_2$ and
$\mathbf{v}_1,\mathbf{v}_2,\mathbf{v}_3$ are minus the rows of $A_1$,
\begin{align}
(\mathbf{u}_1 \times \mathbf{u}_2)\cdot\mathbf{u}_3 &= \det(A_2), \label{u triple}\\
(\mathbf{v}_1 \times \mathbf{v}_2)\cdot\mathbf{v}_3 &= (-1)^3\det(A_1) = -\det(A_1). \label{v triple}
\end{align}
Note the sign with which \eqref{v triple} enters: the term appearing in
\eqref{a dot v} is $-(\mathbf{v}_1\times\mathbf{v}_2)\cdot\mathbf{v}_3 = \det(A_1)$.

Substituting \eqref{four dots}, \eqref{u triple} and \eqref{v triple} into
\eqref{a dot v} and \eqref{b dot u} gives
\begin{align}
\mathbf{a}_{12}\cdot\mathbf{v}_3 &= -(A_0)_{23}(A_1A_2)_{32} - (A_0)_{13}(A_1A_2)_{31} + \det(A_1), \\
\mathbf{b}_{12}\cdot\mathbf{u}_3 &= -(A_0)_{31}(A_1A_2)_{13} - (A_0)_{32}(A_1A_2)_{23} + \det(A_2),
\end{align}
while the two scalar terms in \eqref{trace split} are
\begin{align}
\alpha_{12}(A_0)_{12} &= (A_0)_{12}(A_0)_{23}(A_0)_{31} - (A_0)_{12}(A_1A_2)_{21}, \\
\beta_{12}(A_0)_{21} &= (A_0)_{21}(A_0)_{32}(A_0)_{13} - (A_0)_{21}(A_1A_2)_{12} .
\end{align}
Adding all four, the six mixed terms are precisely the six off-diagonal pairs
$(i,j)$, and we obtain
\begin{multline}\label{trace final}
\tr((x_1 x_2) x_3) = (A_0)_{12}(A_0)_{23}(A_0)_{31} + (A_0)_{21}(A_0)_{32}(A_0)_{13} \\
- \sum_{i \neq j} (A_0)_{ij}(A_1A_2)_{ji} + \det(A_1) + \det(A_2).
\end{multline}

To finish computing $\operatorname{Det}(\Phi(A))$, we substitute into the defining equation \eqref{Det to compute} the trace \eqref{trace final}, the product $\xi_1\xi_2\xi_3 = (A_0)_{11}(A_0)_{22}(A_0)_{33}$, and the sum $\sum_i \xi_i n(x_i)$ of \eqref{xi n sum}. The result splits into three kinds of terms.

\medskip
\noindent\textbf{(1) Terms in $A_0$ alone.}
These are $\xi_1\xi_2\xi_3$, minus the six-term cyclic sum in \eqref{xi n sum},
and the first two terms of \eqref{trace final}. These are exactly the six terms of $\det(A_0)$:
\begin{multline}
(A_0)_{11}(A_0)_{22}(A_0)_{33} - (A_0)_{11}(A_0)_{23}(A_0)_{32} - (A_0)_{22}(A_0)_{31}(A_0)_{13} \\
- (A_0)_{33}(A_0)_{12}(A_0)_{21} + (A_0)_{12}(A_0)_{23}(A_0)_{31} + (A_0)_{21}(A_0)_{32}(A_0)_{13} = \det(A_0).
\end{multline}

\medskip
\noindent\textbf{(2) Mixed terms.}
The diagonal contributions $-\sum_i (A_0)_{ii}(A_1A_2)_{ii}$ come from
$-\sum_i \xi_i n(x_i)$, i.e.\ the negative of \eqref{xi n sum}, and the off-diagonal contributions
$-\sum_{i\neq j} (A_0)_{ij}(A_1A_2)_{ji}$ come from \eqref{trace final}. Together
the restriction $i \neq j$ is lifted and the sum becomes a full contraction:
\begin{equation}
-\sum_{i,j=1}^3 (A_0)_{ij}(A_1A_2)_{ji} = -\tr(A_0A_1A_2).
\end{equation}

\medskip
\noindent\textbf{(3) Determinants of $A_1$ and $A_2$.}
These come from the two scalar triple products in \eqref{trace final}, and by
\eqref{u triple} and \eqref{v triple} they contribute
\begin{equation}
\det(A_1) + \det(A_2).
\end{equation}

\medskip
\noindent Combining the three parts,
\begin{equation}\label{Det computed}
\operatorname{Det}(\Phi(A)) = \det(A_0) + \det(A_1) + \det(A_2) - \tr(A_0A_1A_2) = N(A).
\end{equation}
Thus $\operatorname{Det}(\Phi(A)) = N(A)$ for all $A$, and $\Phi$ is an isomorphism of
Jordan algebras.
\end{proof}


\begin{thebibliography}{[KMRT]}

\bibitem{Ach1} G.\ Achhammer, The first Tits construction of Albert algebras over locally ringed spaces, in \textsl{Nonassociative Algebra and its Applications} (Oviedo, 1993), Math.\ Appl.\ \textbf{303}, Kluwer, Dordrecht, 1994, pp.\ 8--11.

\bibitem{Ach2} G.\ Achhammer, \textsl{Albert-Algebren \"uber lokal geringten R\"aumen}, Ph.D.\ thesis, FernUniversit\"at Hagen, 1995.

\bibitem{Ad} J.~F. Adams, \textit{Lectures on Exceptional Lie Groups}, eds.\ Z.\ Mahmud and M.\ Mimura, Chicago Lectures in Mathematics,
U.\ Chicago Press, Chicago, 1996.


\bibitem{Al21} S.\ Alsaody, Albert algebras over rings and related torsors, \textsl{Canad.\ J.\ Math. }\textbf{73} (2021), 875--898.

\bibitem{Ba} J.\ C.\ Baez, The octonions, \textsl{Bull.\ Amer.\ Math.\ Soc.\ }\textbf{39} (2002), 145–205. Errata in \textsl{Bull.\ Amer.\ Math.\ Soc.\ }\textbf{42} (2005), 213.  Also available as \href{https://arxiv.org/abs/math/0105155}{arXiv:math/0105155}.

\bibitem{BS} J.\ C.\ Baez and P.\ Schwahn, The Standard Model gauge group from the exceptional Jordan algebra, 2026.   Also available at \href{http://arxiv.org/abs/2606.15235}{arXiv:2606.15235}.

\bibitem{BBR} M.\ Barnabei, A.\ Brini and G.-C.\ Rota, On the exterior calculus of invariant theory, \textsl{J.\ Algebra\ }\textbf{96}(1) (1985), 120--160.

\bibitem{Br} P.\ Br\"uhne, \textsl{Ordnungen und die Tits-Konstruktionen von Albert-Algebren}, Ph.D.\ thesis, FernUniversit\"at Hagen, 2000.

\bibitem{Crowe} M.~J.\ Crowe, \textsl{A History of Vector Analysis: The Evolution of the Idea of a Vectorial System}, Dover, New York, 1994.

\bibitem{DM} C.\ Draper and C.\ Mart\'in, Gradings on the Albert algebra and on $\mathfrak{f}_4$, \textsl{Rev.\ Mat.\ Iberoam. }\textbf{25} (2009), 841--908.  Also available as \href{https://arxiv.org/abs/math/0703840}{arXiv:math/0703840}.

\bibitem{GPR} S.\ Garibaldi, H.~P.~Petersson and M.~L.~Racine,
\textsl{Albert Algebras over Commutative Rings: The Last Frontier of Jordan Systems}, Cambridge U.\ Press, Cambridge, 2024.



\bibitem{Joyce} D.~D.\ Joyce, \textsl{Compact Manifolds with Special Holonomy},  Oxford University Press, Oxford, 2000.

\bibitem{KS} N.\ Knarr and M.J.\ Stroppel, Subalgebras of octonion algebras, \textsl{J.\ Algebra\ }\textbf{664} (2025), 42--74. Also available as \href{https://arxiv.org/abs/2303.00335}{arXiv:2303.00335}.


\bibitem{McC69} K.\ McCrimmon, The Freudenthal--Springer--Tits constructions of exceptional Jordan algebras, \textsl{Trans.\ Amer.\ Math.\ Soc. }\textbf{139} (1969), 495--510.

\bibitem{McC04}
K. McCrimmon, \textsl{A Taste of Jordan Algebras}, Springer, Berlin, 2004.

\bibitem{PST} R.\ Parimala, R.\ Sridharan and M.~L.~Thakur, Tits' constructions of Jordan algebras and $F_4$ bundles on the plane, \textsl{Compositio Math. }\textbf{119} (1999), 13--40.

\bibitem{Pe93} H.P.\ Petersson, Composition algebras over algebraic curves of genus zero, \textsl{Trans.\ Amer.\ Math.\ Soc.\ }\textbf{337}(1) (1993), 473--493.

\bibitem{Pe19} H.~P.~Petersson, A survey on Albert algebras, \textsl{Transform.\ Groups\ }  \textbf{24}(1) (2019), 219--278.

\bibitem{PR86} H.~P.~Petersson and M.~L.~Racine, Jordan algebras of degree 3 and the Tits process, \textsl{J.\ Alg.\ }\textbf{98}(1) (1986), 211--243.


\bibitem{Pu98} S.\ Pumpl\"un, Quaternion algebras over elliptic curves, \textsl{Comm.\ Alg.\ }\textbf{26}(12) (1998), 4357--4373.

\bibitem{Pu08} S.\ Pumpl\"un, Albert algebras over curves of genus zero and one, \textsl{J.\ Algebra }\textbf{320}(12) (2008), 4178--4214.  Also available as \href{https://arxiv.org/abs/0709.2308}{arXiv:0709.2308}.

\bibitem{Pu24} S.\ Pumpl\"un, Colour algebras over rings, \textit{Axioms} \textbf{15} (2026), no.\ 2, Paper No.\ 139. \href{https://doi.org/10.3390/axioms15020139}{doi:10.3390/axioms15020139}.  Also available as \href{https://arxiv.org/abs/2409.14574}{arXiv:2409.14574}.


\bibitem{Tits} J.\ Tits, Alg\`ebres alternatives, alg\`ebres de Jordan et alg\`ebres de Lie exceptionnelles.\ I.\ Construction, \textsl{Nederl.\ Akad.\ Wetensch.\ Proc.\ Ser.\ A }\textbf{69} (1966), 223--237.

\bibitem{Y} I.~Yokota, \textsl{Exceptional Lie Groups}, Lecture Notes in Mathematics 2369, Springer, Berlin, 2025.  Also available as \href{https://arxiv.org/abs/0902.0431}{arXiv:0902.0431}.

\bibitem{Z} M.\ Zorn, Alternativk\"orper und quadratische Systeme, \textsl{Abh.\ Math.\ Semin.\ Univ.\ Hambg.\ }\textbf{9}(3--4) (1933), 395--402.

\end{thebibliography}
\end{document}